\documentclass{article}
\usepackage{amsmath, amssymb, amsthm}

\usepackage{geometry}
\usepackage{hyperref}
\usepackage[T1]{fontenc}

\usepackage{indentfirst}
\newtheorem{theorem}{Theorem}[section]

\newtheorem{corollary}{Corollary}[section]

\newtheorem*{theorem*}{Theorem}
\newtheorem*{remark*}{Remark}
\newtheorem*{problem*}{Problem}
\newtheorem*{conjecture*}{Conjecture}
\newtheorem{lemma}[theorem]{Lemma}
\newcommand{\R}{\mathbb{R}}
\usepackage{geometry}
\usepackage{hyperref}
\usepackage[T1]{fontenc}
\usepackage{indentfirst}
\theoremstyle{remark}
\newtheorem{remark}{Remark}
\usepackage{enumitem}

\usepackage{graphicx} 

\usepackage{graphicx} 

\title{Exceptional sets for compositions involving Euler's function,
the sum-of-divisors function and Dedekind's function
}

\author{Aimin Guo$^{*}$}

\date{}

\begin{document}

\maketitle

\vspace{-2.0em}

\begin{center}
\small

\textit{School of Mathematics and Statistics}\\
\textit{Anhui Normal University}\\
\textit{Wuhu 241002, P.~R.~China}

\vspace{0.4em}

$^{*}$Corresponding author:\\
\underline{18339681864@163.com}

\end{center}

\begin{abstract}
Let \(\phi\), \(\psi\), and \(\sigma\) denote Euler's totient function,
Dedekind's arithmetic function, and the sum-of-divisors function,
respectively. We study exceptional sets arising from compositions of
these classical arithmetic functions, with particular emphasis on the
role of small-prime divisibility in the values of the inner functions.
For every fixed \(c>0\), we show that the exceptional sets associated
with \(\phi(\psi(n))\) and \(\phi(\sigma(n))\) have cardinality
\(O_c(x/(\log_3 x)^2)\). In the latter case, this improves the previously
known bound of order \(x/\log_4 x\). We also obtain corresponding
variable-threshold estimates and quantitative density-zero results for
\(\psi(\psi(n))\) and \(\psi(\sigma(n))\), together with consequences
for higher fixed iterates of \(\psi\). Our approach combines
small-prime divisibility with second-moment estimates for the normalized
functions \(\psi(n)/n\) and \(\sigma(n)/n\), providing a common mechanism
for these exceptional-set problems.
\end{abstract}


\maketitle

\bigskip
\section{Introduction}

Let \(\phi\), \(\psi\), and \(\sigma\) denote Euler's totient function,
Dedekind's arithmetic function, and the sum-of-divisors function,
respectively. These are classical multiplicative functions whose
Dirichlet series are closely connected with the Riemann zeta function,
and their average and distributional properties form a classical part
of analytic number theory. Their compositions, however, need not
remain multiplicative and may exhibit substantially more complicated
behavior. Studying such compositions therefore provides a natural way
to investigate how arithmetic information carried by the values of one
function is transformed by another.

Compositions and iterates of arithmetic functions have been studied
from several perspectives; see, for example,
\cite{luca,EGPS,József}. Nevertheless, their distribution remains much
less understood than that of the individual functions. As observed by
Dixit and Bhattacharjee~\cite{Bhattacharjee}, even the question of
whether \(\phi(\sigma(n))\) possesses a normal order remains open, and
a general theory describing the finer distribution of compositions of
multiplicative functions is still lacking. It is therefore natural to
seek not only quantitative estimates for individual compositions, but
also arithmetic mechanisms that persist across different examples.

The study of the exceptional behavior of \(\phi(\sigma(n))\) goes back
to Alaoglu and Erd\H{o}s~\cite{erdos}, who proved that, for every fixed
\(c>0\),
\[
\#\{n\le x:\phi(\sigma(n))\ge cn\}=o(x).
\]
More recently, Dixit and Bhattacharjee~\cite{Bhattacharjee} obtained
the quantitative estimate
\[
\#\{n\le x:\phi(\sigma(n))>cn\}
\le
\frac{\pi^2x}{6c\log_4x}
+
O_c\left(
\frac{x\log_3x}
{(\log x)^{1/\log_3x}\log_4x}
\right),
\]
together with a variable-threshold version for non-decreasing
\(f\) satisfying \(f(x)=o(\log_4x)\). A related problem for
\(\phi(\psi(n))\) was studied by S\'andor
\cite{SÁNDOR2,Sándor}, who proved that
\[
\#\{n\le x:\phi(\psi(n))\ge cn\}=o(x).
\]
These results lead naturally to the question of what arithmetic
structure controls such exceptional sets and whether a common
mechanism can be identified across different compositions.

Our starting point is the small-prime divisibility of the values of
\(\psi\) and \(\sigma\). Apart from
\[
O\left(\frac{x}{(\log_3x)^2}\right)
\]
integers \(n\le x\), these values are divisible by every prime power
below a constant multiple of
\[
\frac{\log_2x}{\log_3x}.
\]
This information interacts directly with the Euler products
\[
\frac{\phi(m)}{m}
=
\prod_{p\mid m}\left(1-\frac1p\right),
\qquad
\frac{\psi(m)}{m}
=
\prod_{p\mid m}\left(1+\frac1p\right).
\]
Indeed, if \(m\) is divisible by every prime below \(y\), then
\[
\frac{\phi(m)}{m}
\le
\prod_{p<y}\left(1-\frac1p\right),
\qquad
\frac{\psi(m)}{m}
\ge
\prod_{p<y}\left(1+\frac1p\right).
\]
Thus the same small-prime structure has two complementary effects:
it suppresses \(\phi(m)/m\), while enhancing \(\psi(m)/m\).
By Mertens' theorem, these products are of orders \(1/\log y\) and
\(\log y\), respectively. Since in the present setting
\[
y\asymp\frac{\log_2x}{\log_3x},
\qquad
\log y\sim\log_3x,
\]
the third iterated logarithm arises naturally from the available range
of small-prime divisibility.

Using this viewpoint together with suitable second-moment estimates,
we prove that, for every fixed \(c>0\),
\[
\#\{n\le x:\phi(\psi(n))\ge cn\}
=
O_c\left(\frac{x}{(\log_3x)^2}\right),
\]
and
\[
\#\{n\le x:\phi(\sigma(n))\ge cn\}
=
O_c\left(\frac{x}{(\log_3x)^2}\right).
\]
The first estimate gives a quantitative refinement of Sándor's
density-zero result, while the second improves the previously known
bound of order \(x/\log_4x\). We also establish corresponding
variable-threshold estimates for non-decreasing functions \(g\) and
\(h\) satisfying
\[
g(x),h(x)=o(\log_3x).
\]

The same mechanism applies in the opposite direction. Since small
prime divisors enlarge \(\psi(m)/m\), we obtain quantitative
density-zero estimates for
\[
\{n\le x:\psi(\psi(n))\le cn\}
\quad\text{and}\quad
\{n\le x:\psi(\sigma(n))\le cn\}.
\]
As a consequence, for every fixed \(k\ge2\),
\[
\#\{n\le x:\psi^{(k)}(n)\le cn\}=o(x).
\]

Taken together, these results show that several apparently different
exceptional-set problems are governed by the same arithmetic
phenomenon: the small-prime structure of the values of the inner
function. It suppresses subsequent composition with \(\phi\), while
enhancing subsequent composition with \(\psi\). Thus our results not
only refine earlier exceptional-set estimates, but also identify a
common mechanism through which arithmetic information in the values of
classical multiplicative functions is reflected in their compositions.

Throughout the paper, \(c>0\) is fixed, and all asymptotic estimates are
understood as \(x\to\infty\). We use the standard \(O\)- and
\(\ll\)-notations, with subscripts indicating the parameters on which
the implied constants may depend. Thus, for example,
\(A\ll_c B\) means \(A=O_c(B)\).

We write
\[
\log_1x=\log x,\qquad
\log_{k+1}x=\log(\log_kx)
\quad(k\ge1).
\]
For an arithmetic function \(f\), let \(f^{(k)}\) denote its \(k\)-fold
iterate, defined recursively by
\[
f^{(0)}(n)=n,\qquad
f^{(k)}(n)=f\bigl(f^{(k-1)}(n)\bigr)
\quad(k\ge1).
\]

We now state the main results of this paper.

\begin{theorem}\label{thm:fixed-c}
For every fixed \(c>0\), we have
\begin{align}
\#\{n\le x:\phi(\psi(n))\ge cn\}
&=
O_c\left(\frac{x}{(\log_3x)^2}\right),
\label{main-thm}\\
\#\{n\le x:\phi(\sigma(n))\ge cn\}
&=
O_c\left(\frac{x}{(\log_3x)^2}\right).
\label{eq:phi-sigma-fixed-c}
\end{align}
In particular, both exceptional sets are \(o(x)\).
\end{theorem}

The estimate \eqref{main-thm} gives a quantitative refinement
of Sándor's density result~\cite{Sándor}, while
\eqref{eq:phi-sigma-fixed-c} improves a result of Dixit and
Bhattacharjee~\cite{Bhattacharjee}. Motivated by the work of Dixit and Bhattacharjee~\cite{Bhattacharjee},
we next replace the fixed thresholds by variable thresholds involving non-decreasing functions. For non-decreasing functions
\(g,h:\R^+ \to \R^+\), define
\[
\begin{cases}
G_g(x) := \left\{ n \leq x : \phi(\psi(n)) \geq \dfrac{n}{g(n)} \right\}, \\[1em]
H_h(x) := \left\{ n \leq x : \phi(\sigma(n)) \geq \dfrac{n}{h(n)} \right\}.
\end{cases}
\]

We have the following result.

\begin{theorem}\label{thm-2}
Suppose that \(g,h:\mathbb{R}^+\to\mathbb{R}^+\) are non-decreasing
functions satisfying
\[
g(x)=o(\log_3x),
\qquad
h(x)=o(\log_3x).
\]
Then, as \(x\to\infty\),
\[
|G_g(x)|
=
O\left(
\frac{(1+g(x)^2)x}{(\log_3x)^2}
\right)
=o(x),
\]
and
\[
|H_h(x)|
=
O\left(
\frac{(1+h(x)^2)x}{(\log_3x)^2}
\right)
=o(x).
\]
\end{theorem}

We next consider exceptional sets associated with compositions
involving Dedekind's arithmetic function.

\begin{theorem}\label{thm-4}

  For every fixed \(c>0\), there exists a constant \(z_c>0\), depending only
on \(c\), such that
\[
\#\bigl\{n\le x:\psi(\psi(n))\le cn\bigr\}
=
O_c\left(
x\left(\frac{\log_2 x}{\log x}\right)^{1/z_c}
\right),
\]
and
\[
\#\bigl\{n\le x:\psi(\sigma(n))\le cn\bigr\}
=
O_c\left(
x\left(\frac{\log_2 x}{\log x}\right)^{1/z_c}
\right).
\]
In particular,
\[
\#\bigl\{n\le x:\psi(\psi(n))\le cn\bigr\}=o(x),
\]
and
\[
\#\bigl\{n\le x:\psi(\sigma(n))\le cn\bigr\}=o(x).
\]
\end{theorem}

\begin{corollary}\label{cor:iterations}
Let $k\geq2$ be fixed. Then, for every fixed $c>0$,
\[
\#\{n\le x:\psi^{(k)}(n)\le cn\}=o(x).
\]
\end{corollary}

\bigskip

\section{\bf Preliminaries}
\bigskip
The proofs of our main results use four main ingredients: product estimates
over primes, sieve estimates for integers avoiding prescribed prime
divisors, results on small-prime divisibility of the values of
\(\psi(n)\) and \(\sigma(n)\), and second-moment estimates for
\(\psi(n)/n\) and \(\sigma(n)/n\). We first record the required auxiliary
results and establish the second-moment estimates needed below. For the
proof of Theorem~\ref{thm-4}, we also establish an upper bound for the
number of integers \(n\le x\) for which \(\psi(n)\) is not divisible by
a prescribed prime. The proof of the latter estimate adapts the sieve
argument of Dixit and Bhattacharjee~\cite{Bhattacharjee}, originally used
in their study of \(\phi(\sigma(n))\).

\begin{lemma}\label{sieve1-lemma}
For sufficiently large \(x\), we have
\[ \log\prod_{p\le x}\left ( 1+\frac{1}{p} \right )  = \log\left ( \frac{6e^{\gamma } }{\pi^{2}}\log x  \right )+o\left ( 1 \right ).\]
In particular, this implies that
\[
\prod_{p\le x}\Bigl(1+\frac1p\Bigr) > \log x.
\]
\end{lemma}

\textbf{Proof.} This result is essentially derived from Littlewood \cite[Lemma 1]{Littlewood}.

\begin{lemma}\label{sieve-lemma} Let $K$ be a set of primes, and for $x > 1$, define \begin{equation*} B(x) = \sum_{\substack{p \leq x \\ p \in K}} \frac{1}{p}. \end{equation*} Then, uniformly for any choice of $K$, the proportion of integers $n \leq x$ that are free of prime factors from $K$ is bounded by $O\left(e^{-B(x)}\right)$. 
\end{lemma}
\textbf{Proof.} This result essentially follows from Pollack and Pomerance \cite[Lemma 3]{pollack}.

\begin{lemma}[Siegel-Walfisz]\label{siegel}
    Let $(b, q) = 1$, and let $\pi(x; q, b)$ denote the count of primes \( p \leq x \) such that \( p \equiv b \pmod{q} \). For any given \( B > 0 \), if \( q \leq (\log x)^B \), then
    \begin{equation*}
        \pi(x; q, b) = \frac{li(x)}{\phi(q)} + O\left(x \exp(-c_3 \sqrt{\log x})\right),
    \end{equation*}
    where the implied constant depends only on \( B \), and \( li(x) := \int_2^x \frac{1}{\log u} \, du \).
   \end{lemma} 
\textbf{Proof.} This result is essentially based on the methods outlined in Multiplicative Number Theory I: Classical Theory \cite[Corollary 11.21]{Montgomery}.

For a prime \(p\), define
\[
T_p(x):=\#\{n\le x:p\nmid\psi(n)\}.
\]
The following estimate is the key ingredient in the proof of
Theorem~\ref{thm-4}.

\begin{lemma}\label{main1-lemma}
    For any prime $p$ and $x\geq e^p$
\begin{equation*}
    T_p(x) = O\left(x \left(\frac{\log\log x}{\log x}\right)^{\frac{1}{p-1}}\right),
\end{equation*}
where the implied constant is absolute.
\end{lemma}

\textbf{Proof.} 
For every prime power \(q^a\), \(a\ge1\), we have
\[
\psi(q^a)=q^{a-1}(q+1).
\]
Hence, if \(q\equiv-1\pmod p\) and \(q\mid n\), then
\[
p\mid\psi(n).
\]
Taking the contrapositive, we obtain
\[
p\nmid\psi(n)
\quad\Longrightarrow\quad
q\nmid n
\]
for every prime \(q\equiv-1\pmod p\).

We first consider the range
\[
e^p\le x\le\exp(e^p).
\]
Since
\[
\log x\le e^p
\qquad\text{and}\qquad
\log_2x\ge\log p\ge\log2,
\]
we have
\[
\left(
\frac{\log_2x}{\log x}
\right)^{\frac1{p-1}}
\ge
\left(
\frac{\log2}{e^p}
\right)^{\frac1{p-1}}.
\]
Moreover,
\[
\left(
\frac{e^p}{\log2}
\right)^{\frac1{p-1}}
=
e^{\frac{p}{p-1}}
(\log2)^{-\frac1{p-1}}
\le
\frac{e^2}{\log2}<11.
\]
Therefore, using the trivial estimate \(T_p(x)\le x\), we obtain
\[
T_p(x)
\le
11x
\left(
\frac{\log_2x}{\log x}
\right)^{\frac1{p-1}}.
\]

We may now assume that
\[
\log x>e^p.
\]
Define
\[
\mathcal K_p(x)
:=
\left\{
q:\ q\text{ is prime},\
\log x<q<x,\
q\equiv-1\pmod p
\right\}.
\]
By the contrapositive established above,
\[
\{n\le x:p\nmid\psi(n)\}
\subseteq
\left\{
n\le x:
q\nmid n\text{ for every }q\in\mathcal K_p(x)
\right\}.
\]

Since \(\log x>e^p\), we have \(p<\log t\) for every
\(t\in[\log x,x]\). Hence Lemma~\ref{siegel}, with \(B=1\),
applies uniformly throughout this interval. By partial summation,
\[
\begin{aligned}
\sum_{q\in\mathcal K_p(x)}\frac1q
&=
\frac{\pi(x;p,-1)}x
-
\frac{\pi(\log x;p,-1)}{\log x}
+
\int_{\log x}^x
\frac{\pi(t;p,-1)}{t^2}\,dt\\
&=
\frac1{p-1}
\int_{\log x}^x\frac{dt}{t\log t}
+
O\left(\frac1{\log_2x}\right)\\
&=
\frac{\log_2x-\log_3x}{p-1}
+
O\left(\frac1{\log_2x}\right).
\end{aligned}
\]

Applying Lemma~\ref{sieve-lemma} with
\(\mathcal K_p(x)\) as the set of sieving primes, we obtain
\[
\begin{aligned}
T_p(x)
&\ll
x\exp\left(
-\sum_{q\in\mathcal K_p(x)}\frac1q
\right)\ll
x\exp\left(
-\frac{\log_2x-\log_3x}{p-1}
\right)=
x\left(
\frac{\log_2x}{\log x}
\right)^{\frac1{p-1}}.
\end{aligned}
\]
This completes the proof.

\begin{lemma}\label{div-lemma}

 Suppose that \(x\) is sufficiently large. Then there exists
a constant \(c_1>0\) such that, for all positive integers \(n\le x\),
with \(O(x/(\log_3 x)^2)\) exceptions, \(\psi(n)\) and \(\phi(n)\) are both divisible by every prime power \(p^a\) satisfying
\[
p^a<c_1\frac{\log_2 x}{\log_3 x}.
\]
\end{lemma}

\textbf{Proof.} The result in Guo et al.~\cite{Guo} is stated for \(n<x\).
If \(x\) is an integer, the additional integer \(n=x\) may be included
in the exceptional set. Since this adds at most one exception, the
stated error term is unchanged.

\begin{lemma}\label{div-lemma1}
 There exists a positive constant \(c_2\) such that, for
sufficiently large \(x\), \(\phi(n)\) and \(\sigma(n)\) are both
divisible by every prime power
\[
p^a<c_2\frac{\log_2 x}{\log_3 x}
\]
for all positive integers \(n\le x\), with
\(O(x/(\log_3 x)^2)\) exceptions.

\end{lemma}

\textbf{Proof.} The result in \cite[Theorem 4]{ref2} is stated for \(n<x\). As discussed above, the possible endpoint \(n=x\) can be absorbed into the exceptional set.

\begin{lemma}\label{lem:second-moments}
There exist absolute constants $C_{\psi},C_{\sigma}>0$ such that, for every
$x\geq 1$,
\[
\sum_{n\leq x}\left(\frac{\psi(n)}{n}\right)^2
\leq C_{\psi}x \qquad
and
\qquad
\sum_{n\leq x}\left(\frac{\sigma(n)}{n}\right)^2
\leq C_{\sigma}x.
\]
In particular,
\[
\sum_{n\leq x}\left(\frac{\psi(n)}{n}\right)^2\ll x,
\qquad
\sum_{n\leq x}\left(\frac{\sigma(n)}{n}\right)^2\ll x.
\]
\end{lemma}

\begin{proof}
We first consider the Dedekind function. Since
\[
\frac{\psi(n)}{n}
=
\prod_{p\mid n}\left(1+\frac1p\right),
\]
we have
\[
\left(\frac{\psi(n)}{n}\right)^2
=
\prod_{p\mid n}
\left(1+\frac2p+\frac1{p^2}\right).
\]
Define the multiplicative function \(a\) by
\[
a(1)=1,\qquad
a(p)=\frac2p+\frac1{p^2},
\qquad
a(p^\nu)=0\quad(\nu\geq2).
\]
Then
\[
\left(\frac{\psi(n)}{n}\right)^2
=
\sum_{d\mid n}a(d).
\]
Since \(a(d)\geq0\), it follows that
\[
\begin{aligned}
\sum_{n\leq x}\left(\frac{\psi(n)}{n}\right)^2
&=
\sum_{n\leq x}\sum_{d\mid n}a(d)\\
&=
\sum_{d\leq x}a(d)
\left\lfloor\frac{x}{d}\right\rfloor\\
&\leq
x\sum_{d\leq x}\frac{a(d)}{d}\\
&\leq
x\sum_{d=1}^{\infty}\frac{a(d)}{d}.
\end{aligned}
\]

We now show that the series on the right converges. For \(z\geq2\),
multiplicativity gives
\[
\sum_{\substack{d\geq1\\ p\mid d\Rightarrow p\leq z}}
\frac{a(d)}{d}
=
\prod_{p\leq z}
\left(1+\frac{a(p)}{p}\right)
=
\prod_{p\leq z}
\left(1+\frac2{p^2}+\frac1{p^3}\right).
\]
Put
\[
u_p:=\frac2{p^2}+\frac1{p^3}.
\]
Since \(p\geq2\), we have
\[
0<u_p
=
\frac2{p^2}+\frac1{p^3}
\leq
\frac3{p^2}.
\]
Moreover,
\[
\sum_p\frac1{p^2}
\leq
\sum_{n=2}^{\infty}\frac1{n^2}
=
\frac{\pi^2}{6}-1
<
\infty.
\]
Hence
\[
\sum_pu_p
\leq
3\sum_p\frac1{p^2}
<
\infty.
\]
Using the elementary inequality
\[
\log(1+t)\leq t
\qquad(t\geq0),
\]
we obtain
\[
0\leq
\sum_p\log(1+u_p)
\leq
\sum_pu_p
<
\infty.
\]
Therefore the Euler product
\[
\prod_p
\left(1+\frac2{p^2}+\frac1{p^3}\right)
\]
converges to a finite positive value.

Since \(a(d)\geq0\), letting \(z\to\infty\) in the finite Euler-product
identity above and using monotone convergence yields
\[
\sum_{d=1}^{\infty}\frac{a(d)}{d}
=
\prod_p
\left(1+\frac2{p^2}+\frac1{p^3}\right)
<\infty.
\]
We may therefore define
\[
C_\psi
:=
\sum_{d=1}^{\infty}\frac{a(d)}{d}
=
\prod_p
\left(1+\frac2{p^2}+\frac1{p^3}\right)
<\infty.
\]
Consequently,
\[
\sum_{n\leq x}
\left(\frac{\psi(n)}{n}\right)^2
\leq C_\psi x.
\]

We next consider the sum-of-divisors function. Put
\[
F(n):=\left(\frac{\sigma(n)}{n}\right)^2.
\]
Since \(\sigma(n)/n\) is multiplicative, so is \(F\). Define
\[
b:=F*\mu.
\]
Then \(b\) is multiplicative and
\[
F(n)=\sum_{d\mid n}b(d).
\]
For every prime power \(p^\nu\), \(\nu\geq1\), we have
\[
\frac{\sigma(p^\nu)}{p^\nu}
=
1+\frac1p+\cdots+\frac1{p^\nu}.
\]
Hence
\[
b(p^\nu)
=
\left(1+\frac1p+\cdots+\frac1{p^\nu}\right)^2
-
\left(1+\frac1p+\cdots+\frac1{p^{\nu-1}}\right)^2.
\]
In particular, \(b(p^\nu)\geq0\), and therefore \(b(d)\geq0\) for all \(d\).

Let
\[
S_\nu:=1+\frac1p+\cdots+\frac1{p^\nu}.
\]
Then
\[
b(p^\nu)
=
S_\nu^2-S_{\nu-1}^2
=
(S_\nu-S_{\nu-1})(S_\nu+S_{\nu-1}).
\]
Since
\[
S_\nu-S_{\nu-1}=\frac1{p^\nu}
\]
and
\[
S_\nu,S_{\nu-1}
\leq
\sum_{j=0}^{\infty}\frac1{p^j}
=
\frac{p}{p-1}
\leq2,
\]
we obtain
\[
0\leq b(p^\nu)\leq\frac4{p^\nu}.
\]

Since \(b(d)\geq0\), it follows that
\[
\begin{aligned}
\sum_{n\leq x}\left(\frac{\sigma(n)}{n}\right)^2
&=
\sum_{n\leq x}\sum_{d\mid n}b(d)\\
&=
\sum_{d\leq x}b(d)
\left\lfloor\frac{x}{d}\right\rfloor\\
&\leq
x\sum_{d\leq x}\frac{b(d)}{d}\\
&\leq
x\sum_{d=1}^{\infty}\frac{b(d)}{d}.
\end{aligned}
\]

We now show that the series on the right converges. For \(z\geq2\),
multiplicativity gives
\[
\sum_{\substack{d\geq1\\ p\mid d\Rightarrow p\leq z}}
\frac{b(d)}{d}
=
\prod_{p\leq z}
\left(
1+\sum_{\nu\geq1}\frac{b(p^\nu)}{p^\nu}
\right).
\]
Put
\[
v_p:=
\sum_{\nu\geq1}\frac{b(p^\nu)}{p^\nu}.
\]
By the estimate above,
\[
0\leq v_p
\leq
4\sum_{\nu\geq1}\frac1{p^{2\nu}}
=
\frac4{p^2-1}.
\]
Since \(p\geq2\),
\[
\frac1{p^2-1}\leq\frac2{p^2},
\]
and hence
\[
\sum_pv_p
\leq
8\sum_p\frac1{p^2}
\leq
8\sum_{n=2}^{\infty}\frac1{n^2}
=
8\left(\frac{\pi^2}{6}-1\right)
<
\infty.
\]
Again using
\[
\log(1+t)\leq t
\qquad(t\geq0),
\]
we obtain
\[
0\leq
\sum_p\log(1+v_p)
\leq
\sum_pv_p
<
\infty.
\]
Therefore the Euler product
\[
\prod_p(1+v_p)
=
\prod_p
\left(
1+\sum_{\nu\geq1}\frac{b(p^\nu)}{p^\nu}
\right)
\]
converges to a finite positive value.

Since \(b(d)\geq0\), letting \(z\to\infty\) in the finite Euler-product
identity above and using monotone convergence yields
\[
\sum_{d=1}^{\infty}\frac{b(d)}{d}
=
\prod_p
\left(
1+\sum_{\nu\geq1}\frac{b(p^\nu)}{p^\nu}
\right)
<\infty.
\]
We may therefore define
\[
C_\sigma
:=
\sum_{d=1}^{\infty}\frac{b(d)}{d}
=
\prod_p
\left(
1+\sum_{\nu\geq1}\frac{b(p^\nu)}{p^\nu}
\right)
<\infty.
\]
Consequently,
\[
\sum_{n\leq x}
\left(\frac{\sigma(n)}{n}\right)^2
\leq C_\sigma x.
\]
This completes the proof.
\end{proof}

\medskip

\section{Proof of Theorem~\ref{thm:fixed-c} }
Put
\[
        y_x=c_1\frac{\log_2 x}{\log_3 x},
\]
where \(c_1\) is the constant appearing in Lemma~\ref{div-lemma}, and define
\[
        P(y_x):=\prod_{p<y_x}p .
\]
By Lemma~\ref{div-lemma}, there exists an exceptional set
\(\mathcal E(x)\subseteq [1,x]\) such that
\[
        |\mathcal E(x)|
        =O\left(\frac{x}{(\log_3 x)^2}\right),
\]
and, for every \(n\le x\) with \(n\notin\mathcal E(x)\), the integer
\(\psi(n)\) is divisible by every prime power \(p^a<y_x\). In particular,
every prime \(p<y_x\) divides \(\psi(n)\). Hence,
\[
        P(y_x)\mid \psi(n)
        \qquad (n\le x,\ n\notin\mathcal E(x)).
\]

Therefore, for \(n\le x\) with \(n\notin\mathcal E(x)\), we have
\[
\begin{aligned}
 \phi(\psi(n))
 &=\psi(n)\prod_{p\mid \psi(n)}
        \left(1-\frac1p\right)\\
 &\leq \psi(n)\prod_{p<y_x}
        \left(1-\frac1p\right).
\end{aligned}
\]
By Mertens' theorem, for sufficiently large \(x\),
\[
        \prod_{p<y_x}\left(1-\frac1p\right)
        <\frac1{\log y_x}.
\]
Consequently,
\[
        \phi(\psi(n))
        <\frac{\psi(n)}{\log y_x}
        \qquad (n\le x,\ n\notin\mathcal E(x)).
\]

Let
\[
        \mathcal A_c(x):=
        \{n\le x:\phi(\psi(n))\ge cn\}.
\]
If
\[
        n\in\mathcal A_c(x)\setminus\mathcal E(x),
\]
then
\[
        cn\leq \phi(\psi(n))
        <\frac{\psi(n)}{\log y_x},
\]
and hence
\[
        \frac{\psi(n)}n>c\log y_x .
\]

Now put
\[
        \mathcal B_c(x):=
        \left\{
        n\le x:\frac{\psi(n)}n>c\log y_x
        \right\}.
\]
Then
\[
        \mathcal A_c(x)\setminus\mathcal E(x)
        \subseteq \mathcal B_c(x),
\]
and therefore
\[
        |\mathcal A_c(x)|
        \leq |\mathcal E(x)|+|\mathcal B_c(x)|.
\]

For every \(n\in\mathcal B_c(x)\), we have
\[
        \frac{\psi(n)}n>c\log y_x.
\]
Since both sides are positive, squaring gives
\[
        \left(\frac{\psi(n)}n\right)^2
        >c^2(\log y_x)^2.
\]
Summing this inequality over all \(n\in\mathcal B_c(x)\), we obtain
\[
\begin{aligned}
 \sum_{n\in\mathcal B_c(x)}
 \left(\frac{\psi(n)}n\right)^2
 &>
 \sum_{n\in\mathcal B_c(x)}
 c^2(\log y_x)^2\\
 &=
 c^2(\log y_x)^2|\mathcal B_c(x)|.
\end{aligned}
\]
On the other hand, since all the summands are non-negative and
\(\mathcal B_c(x)\subseteq [1,x]\),
\[
        \sum_{n\in\mathcal B_c(x)}
        \left(\frac{\psi(n)}n\right)^2
        \leq
        \sum_{n\le x}
        \left(\frac{\psi(n)}n\right)^2.
\]
Hence
\[
        c^2(\log y_x)^2|\mathcal B_c(x)|
        <
        \sum_{n\le x}
        \left(\frac{\psi(n)}n\right)^2,
\]
and therefore
\[
        |\mathcal B_c(x)|
        \leq
        \frac1{c^2(\log y_x)^2}
        \sum_{n\le x}
        \left(\frac{\psi(n)}n\right)^2.
\]
It follows that
\[
\begin{aligned}
 |\mathcal A_c(x)|
 &\leq
 |\mathcal E(x)|
 +
 \frac1{c^2(\log y_x)^2}
 \sum_{n\le x}
 \left(\frac{\psi(n)}n\right)^2.
\end{aligned}
\]

By Lemma~\ref{lem:second-moments},
\[
        \sum_{n\le x}
        \left(\frac{\psi(n)}n\right)^2
        \ll x.
\]
Hence
\[
\begin{aligned}
 |\mathcal A_c(x)|
 &\ll
 \frac{x}{(\log_3x)^2}
 +
 \frac{x}{c^2(\log y_x)^2}.
\end{aligned}
\]

Finally,
\[
\begin{aligned}
 \log y_x
 &=\log\left(c_1\frac{\log_2x}{\log_3x}\right)\\
 &=\log_3x-\log_4x+O(1)\\
 &=\log_3x
 \left(
 1-\frac{\log_4x}{\log_3x}
 +\frac{O(1)}{\log_3x}
 \right)\\
 &=(1+o(1))\log_3x .
\end{aligned}
\]
Therefore,
\[
        \frac1{(\log y_x)^2}
        \asymp
        \frac1{(\log_3x)^2},
\]
and consequently
\[
        \#\{n\le x:\phi(\psi(n))\ge cn\}
        \ll_c
        \frac{x}{(\log_3x)^2}.
\]
In particular,
\[
        \#\{n\le x:\phi(\psi(n))\ge cn\}=o(x).
\]

The estimate for \(\phi(\sigma(n))\) is obtained in the same way.
By Lemma~\ref{div-lemma1}, there exists an exceptional set
\(\mathcal E_\sigma(x)\subseteq[1,x]\) such that
\[
        |\mathcal E_\sigma(x)|
        =
        O\left(\frac{x}{(\log_3x)^2}\right),
\]
and, for every \(n\le x\) with
\(n\notin\mathcal E_\sigma(x)\), the integer \(\sigma(n)\) is divisible
by every prime power
\[
        p^a<c_2\frac{\log_2x}{\log_3x}.
\]
Put
\[
        y_x=c_2\frac{\log_2x}{\log_3x}.
\]
Then every prime \(p<y_x\) divides \(\sigma(n)\), and hence
\[
        P(y_x)\mid\sigma(n)
        \qquad
        (n\le x,\ n\notin\mathcal E_\sigma(x)).
\]
Therefore,
\[
\begin{aligned}
 \phi(\sigma(n))
 &=
 \sigma(n)
 \prod_{p\mid\sigma(n)}
 \left(1-\frac1p\right)\\
 &\leq
 \sigma(n)
 \prod_{p<y_x}
 \left(1-\frac1p\right)\\
 &<
 \frac{\sigma(n)}{\log y_x},
\end{aligned}
\]
for sufficiently large \(x\), by Mertens' theorem.

Let
\[
        \mathcal C_c(x)
        :=
        \{n\le x:\phi(\sigma(n))\ge cn\}.
\]
If
\[
        n\in
        \mathcal C_c(x)\setminus\mathcal E_\sigma(x),
\]
then
\[
        cn
        \leq
        \phi(\sigma(n))
        <
        \frac{\sigma(n)}{\log y_x},
\]
and hence
\[
        \frac{\sigma(n)}n
        >
        c\log y_x.
\]

Put
\[
        \mathcal D_c(x)
        :=
        \left\{
        n\le x:
        \frac{\sigma(n)}n>c\log y_x
        \right\}.
\]
Then
\[
        \mathcal C_c(x)\setminus\mathcal E_\sigma(x)
        \subseteq
        \mathcal D_c(x),
\]
so that
\[
        |\mathcal C_c(x)|
        \leq
        |\mathcal E_\sigma(x)|
        +
        |\mathcal D_c(x)|.
\]

For every \(n\in\mathcal D_c(x)\),
\[
        \frac{\sigma(n)}n>c\log y_x,
\]
and hence
\[
        \left(\frac{\sigma(n)}n\right)^2
        >
        c^2(\log y_x)^2.
\]
Summing over \(n\in\mathcal D_c(x)\), we obtain
\[
\begin{aligned}
 \sum_{n\in\mathcal D_c(x)}
 \left(\frac{\sigma(n)}n\right)^2
 &>
 c^2(\log y_x)^2
 |\mathcal D_c(x)|.
\end{aligned}
\]
Since all the summands are non-negative,
\[
        \sum_{n\in\mathcal D_c(x)}
        \left(\frac{\sigma(n)}n\right)^2
        \leq
        \sum_{n\le x}
        \left(\frac{\sigma(n)}n\right)^2.
\]
Therefore,
\[
        |\mathcal D_c(x)|
        \leq
        \frac1{c^2(\log y_x)^2}
        \sum_{n\le x}
        \left(\frac{\sigma(n)}n\right)^2.
\]
Consequently,
\[
|C_c(x)|
\leq
|\mathcal E_\sigma(x)|
+
\frac{1}{c^2(\log y_x)^2}
\sum_{n\leq x}
\left(\frac{\sigma(n)}{n}\right)^2.
\]

By Lemma~\ref{lem:second-moments},
\[
        \sum_{n\le x}
        \left(\frac{\sigma(n)}n\right)^2
        \ll x.
\]
Thus
\[
        |\mathcal C_c(x)|
        \ll_c
        \frac{x}{(\log_3x)^2}
        +
        \frac{x}{(\log y_x)^2}.
\]
Since
\[
        \log y_x=(1+o(1))\log_3x,
\]
we conclude that
\[
        \#\{n\le x:\phi(\sigma(n))\ge cn\}
        \ll_c
        \frac{x}{(\log_3x)^2}.
\]
In particular,
\[
        \#\{n\le x:\phi(\sigma(n))\ge cn\}=o(x).
\]

\section{Proof of Theorem~\ref{thm-2}}

The proof is similar to that of Theorem~\ref{thm:fixed-c}.
We give the details needed for the second-moment argument.

Let
\[
        y_x=c_1\frac{\log_2x}{\log_3x},
\]
where \(c_1\) is the constant appearing in Lemma~\ref{div-lemma}.
By that lemma, there exists an exceptional set
\(\mathcal E(x)\subseteq[1,x]\) such that
\[
        |\mathcal E(x)|
        =
        O\left(\frac{x}{(\log_3x)^2}\right),
\]
and, outside \(\mathcal E(x)\),
\[
        P(y_x)\mid\psi(n).
\]
Hence, by Mertens' theorem,
\[
        \phi(\psi(n))
        <
        \frac{\psi(n)}{\log y_x}
        \qquad
        (n\le x,\ n\notin\mathcal E(x)).
\]

Recall that
\[
        G_g(x)
        :=
        \left\{
        n\le x:
        \phi(\psi(n))\geq\frac{n}{g(n)}
        \right\}.
\]
Suppose that
\[
        n\in G_g(x)\setminus\mathcal E(x).
\]
Then
\[
        \frac{n}{g(n)}
        \leq
        \phi(\psi(n))
        <
        \frac{\psi(n)}{\log y_x}.
\]
Hence
\[
        \frac{\psi(n)}n
        >
        \frac{\log y_x}{g(n)}.
\]
Since \(g\) is non-decreasing and \(n\le x\),
\[
        g(n)\le g(x),
\]
so that
\[
        \frac1{g(n)}
        \geq
        \frac1{g(x)}.
\]
Therefore
\[
        \frac{\psi(n)}n
        >
        \frac{\log y_x}{g(x)}.
\]

Define
\[
        \mathcal B_g(x)
        :=
        \left\{
        n\le x:
        \frac{\psi(n)}n
        >
        \frac{\log y_x}{g(x)}
        \right\}.
\]
Then
\[
        G_g(x)\setminus\mathcal E(x)
        \subseteq
        \mathcal B_g(x),
\]
and hence
\[
        |G_g(x)|
        \leq
        |\mathcal E(x)|
        +
        |\mathcal B_g(x)|.
\]

For each \(n\in\mathcal B_g(x)\),
\[
        \frac{\psi(n)}n
        >
        \frac{\log y_x}{g(x)}.
\]
Squaring gives
\[
        \left(\frac{\psi(n)}n\right)^2
        >
        \frac{(\log y_x)^2}{g(x)^2}.
\]
Summing over all \(n\in\mathcal B_g(x)\), we obtain
\[
\begin{aligned}
 \sum_{n\in\mathcal B_g(x)}
 \left(\frac{\psi(n)}n\right)^2
 &>
 \frac{(\log y_x)^2}{g(x)^2}
 |\mathcal B_g(x)|.
\end{aligned}
\]
On the other hand,
\[
        \sum_{n\in\mathcal B_g(x)}
        \left(\frac{\psi(n)}n\right)^2
        \leq
        \sum_{n\le x}
        \left(\frac{\psi(n)}n\right)^2.
\]
Therefore
\[
        |\mathcal B_g(x)|
        \leq
        \frac{g(x)^2}{(\log y_x)^2}
        \sum_{n\le x}
        \left(\frac{\psi(n)}n\right)^2.
\]
Consequently,
\[
\begin{aligned}
 |G_g(x)|\le
 &|\mathcal E(x)|
 +
 \frac{g(x)^2}{(\log y_x)^2}
 \sum_{n\le x}
 \left(\frac{\psi(n)}n\right)^2.
\end{aligned}
\]
By Lemma~\ref{lem:second-moments},
\[
        \sum_{n\le x}
        \left(\frac{\psi(n)}n\right)^2
        \ll x.
\]
Hence
\[
        |G_g(x)|
        \ll
        \frac{x}{(\log_3x)^2}
        +
        \frac{g(x)^2x}{(\log y_x)^2}.
\]
Since
\[
        \log y_x=(1+o(1))\log_3x,
\]
we obtain
\[
        |G_g(x)|
        \ll
        \frac{x}{(\log_3x)^2}
        +
        \frac{g(x)^2x}{(\log_3x)^2}.
\]
Equivalently,
\[
        |G_g(x)|
        \ll
        \frac{(1+g(x)^2)x}{(\log_3x)^2}.
\]
Since
\[
        g(x)=o(\log_3x),
\]
we have
\[
        \frac{g(x)^2}{(\log_3x)^2}=o(1),
\]
and therefore
\[
        |G_g(x)|=o(x).
\]

The estimate for \(H_h(x)\) is obtained in the same way. Put
\[
        y_x=c_2\frac{\log_2x}{\log_3x}.
\]
By Lemma~\ref{div-lemma1}, outside an exceptional set
\(\mathcal E_\sigma(x)\) satisfying
\[
        |\mathcal E_\sigma(x)|
        =
        O\left(\frac{x}{(\log_3x)^2}\right),
\]
we have
\[
        \phi(\sigma(n))
        <
        \frac{\sigma(n)}{\log y_x}.
\]

Recall that
\[
        H_h(x)
        :=
        \left\{
        n\le x:
        \phi(\sigma(n))
        \geq
        \frac{n}{h(n)}
        \right\}.
\]
If
\[
        n\in H_h(x)\setminus\mathcal E_\sigma(x),
\]
then
\[
        \frac{\sigma(n)}n
        >
        \frac{\log y_x}{h(n)}
        \geq
        \frac{\log y_x}{h(x)}.
\]
Define
\[
        \mathcal D_h(x)
        :=
        \left\{
        n\le x:
        \frac{\sigma(n)}n
        >
        \frac{\log y_x}{h(x)}
        \right\}.
\]
Then
\[
        H_h(x)\setminus\mathcal E_\sigma(x)
        \subseteq
        \mathcal D_h(x),
\]
and hence
\[
        |H_h(x)|
        \leq
        |\mathcal E_\sigma(x)|
        +
        |\mathcal D_h(x)|.
\]

For every \(n\in\mathcal D_h(x)\),
\[
        \left(\frac{\sigma(n)}n\right)^2
        >
        \frac{(\log y_x)^2}{h(x)^2}.
\]
Therefore,
\[
\begin{aligned}
 \frac{(\log y_x)^2}{h(x)^2}
 |\mathcal D_h(x)|
 &<
 \sum_{n\in\mathcal D_h(x)}
 \left(\frac{\sigma(n)}n\right)^2\\
 &\leq
 \sum_{n\le x}
 \left(\frac{\sigma(n)}n\right)^2.
\end{aligned}
\]
Thus
\[
        |\mathcal D_h(x)|
        \leq
        \frac{h(x)^2}{(\log y_x)^2}
        \sum_{n\le x}
        \left(\frac{\sigma(n)}n\right)^2.
\]
Using Lemma~\ref{lem:second-moments}, we obtain
\[
\begin{aligned}
 |H_h(x)|
 &\ll
 \frac{x}{(\log_3x)^2}
 +
 \frac{h(x)^2x}{(\log y_x)^2}\\
 &\ll
 \frac{x}{(\log_3x)^2}
 +
 \frac{h(x)^2x}{(\log_3x)^2}.
\end{aligned}
\]
Therefore
\[
        |H_h(x)|
        \ll
        \frac{(1+h(x)^2)x}{(\log_3x)^2}.
\]
Since
\[
        h(x)=o(\log_3x),
\]
it follows that
\[
        |H_h(x)|=o(x).
\]

\section{\textbf{Proof of Theorem \ref{thm-4} }and Corollary~\ref{cor:iterations}}
\medskip

Note that

\begin{equation*}               \psi\left(\psi\left(n\right)\right)=\psi\left(n\right)\prod_{p \mid \psi\left(n\right)} \left(1 + \frac{1}{p}\right)
\end{equation*}
Denote by $P(z)= \prod\limits_{p\leq z} p$ , the product of all primes $\leq z$. If $P(z) | \psi(n)$, then

\begin{align*}
      \psi(\psi(n)) &= \psi(n) \prod_{p|\psi(n)} \left(1+\frac{1}{p}\right)\\
      &\geq \psi(n) \prod_{p\leq z} \left(1+ \frac{1}{p}\right)  > \psi(n)\log z \geq n\log z ,
\end{align*}
where the strict inequality follows from Lemma~\ref{sieve1-lemma},
and the last inequality follows from \(\psi(n)\ge n\).

Thus, for any $c>0$, $\psi(\psi(n))> cn$ if $P(z)\mid\psi(n)$ and $\log z > c$. Take
\[
z_c=\max\{\lfloor e^c\rfloor+1,C_0\},
\]
where $C_0$ is a sufficiently large absolute constant ensuring that
Lemma~\ref{sieve1-lemma} holds for all $z\ge C_0$. This choice
guarantees both $\log z_c>c$ and the applicability of the lemma.
Taking the contrapositive yields
\[
\psi(\psi(n))\le cn \quad\Longrightarrow\quad P(z_c)\nmid\psi(n).
\]
Hence,
\[
\#\{n\le x:\psi(\psi(n))\le cn\}
\le
\#\{n\le x:P(z_c)\nmid\psi(n)\}.
\]
By Lemma~\ref{main1-lemma},
\[
\#\{n\le x:P(z_c)\nmid\psi(n)\}
\le
\sum_{p\le z_c}T_p(x)
=
O_c\left(
x\left(\frac{\log_2 x}{\log x}\right)^{1/z_c}
\right).
\]
Therefore,
\[
\#\{n\le x:\psi(\psi(n))\le cn\}
\ll_c
x\left(\frac{\log_2 x}{\log x}\right)^{1/z_c}
=o(x).
\]

The same argument applies to \(\psi(\sigma(n))\). Indeed, if
\(P(z_c)\mid \sigma(n)\), then, since \(\sigma(n)\ge n\),
\[
\psi(\sigma(n))
\geq \sigma(n)\prod_{p\le z_c}\left(1+\frac1p\right)
>
n\log z_c
>
cn.
\]
Consequently,
\[
\{n\le x:\psi(\sigma(n))\le cn\}
\subseteq
\{n\le x:P(z_c)\nmid\sigma(n)\}.
\]
For a prime \(p\), let
\[
S_p(x):=\#\{n\le x:p\nmid\sigma(n)\}.
\]
By the sieve estimate of Dixit and Bhattacharjee~\cite{Bhattacharjee},
\[
S_p(x)
=O\!\left(
x\left(\frac{\log_2 x}{\log x}\right)^{1/(p-1)}
\right).
\]
Since \(z_c\) is fixed, we obtain
\[
\begin{aligned}
\#\{n\le x:P(z_c)\nmid\sigma(n)\}
&\le \sum_{p\le z_c}S_p(x)\\
&\ll_c
x\left(\frac{\log_2 x}{\log x}\right)^{1/z_c}.
\end{aligned}
\]
Combining this with the preceding inclusion, we obtain
\[
\#\{n\le x:\psi(\sigma(n))\le cn\}
\ll_c
x\left(\frac{\log_2 x}{\log x}\right)^{1/z_c}
=o(x).
\]
This completes the proof of Theorem~\ref{thm-4}.

For $\psi$, note that $\psi(m)\ge m$ for all $m$, hence
\[
\psi^{(k)}(n)\ge\psi^{(k-1)}(n)\ge\cdots\ge\psi^{(2)}(n)=\psi(\psi(n)).
\]
Thus 
\[\{n\le x:\psi^{(k)}(n)\le cn\}\subseteq\{n\le x:\psi(\psi(n))\le cn\},\] 
and the latter set has size $o(x)$ by Theorem~\ref{thm-4}. This completes the proof of Corollary~\ref{cor:iterations}.

\section{\bf Concluding Remark}
\medskip

\begin{remark}
The bounds in Theorem~\ref{thm:fixed-c} are not claimed to be optimal.
The second-moment argument gives an upper bound of order
\[
\frac{x}{(\log_3x)^2}
\]
for the relevant large-value sets. This is already of the same order
as the exceptional sets arising from Lemmas~\ref{div-lemma} and
\ref{div-lemma1}. Consequently, within the present approach, replacing
the second moment by a higher fixed moment would not by itself improve
the final order of magnitude. Any further improvement would require,
for example, a sharper small-prime divisibility estimate or a more
direct use of the arithmetic structure of the exceptional sets.
\end{remark}

\section*{Acknowledgments}
The author thanks Kannan Soundararajan for kindly pointing out the
counterexample \(n=39270\) to S\'andor's conjecture
\(\phi(\psi(n))\le n\).

\end{document}